\documentclass[a4paper, oneside, 12pt]{amsart}
\usepackage{graphicx}
\usepackage[left=1.2cm,right=1.5cm,top=2cm,bottom=3cm]{geometry}
\usepackage{eulervm}
\usepackage{amsfonts}
\usepackage{amsmath}
\usepackage{amssymb}
\usepackage{amsthm}
\usepackage{enumitem}
\usepackage{tikz-cd}
\usepackage{dynkin-diagrams}
\usepackage{hyperref}
\hypersetup{
  colorlinks = true,
  citecolor = {purple!70!black},
  linkcolor = {purple!70!black},
  urlcolor = {purple!70!black}
}
\usepackage{import}
\usepackage{thmtools}
\usepackage{thm-restate}
\declaretheoremstyle[
headformat=\NAME\,\NUMBER\NOTE,
postheadspace=.5em,
spaceabove=6pt,
headfont=\bfseries,
notefont=\normalfont\small\mdseries, notebraces={(}{)},
bodyfont=\normalfont\itshape
]{plainswap}
\declaretheorem[style=plainswap, name=Theorem, sharenumber=subsection]{theorem}

\declaretheorem[style=plainswap, numberlike=theorem, name=Lemma]{lemma}

\declaretheorem[style=plainswap, numbered=no, name=Main Theorem]{mainthm}

\declaretheoremstyle[
headformat=\NAME\,\NUMBER\NOTE,
postheadspace=.5em,
spaceabove=6pt,
headfont=\bfseries,
notefont=\normalfont\mdseries, notebraces={(}{)},
bodyfont=\normalfont
]{definitionswap}
\theoremstyle{definition}
\declaretheorem[style=definitionswap, numberlike=theorem, name=Definition]{definition}
\declaretheorem[style=definitionswap, numberlike=theorem, name=Example]{example}
\declaretheorem[style=definitionswap, numberlike=theorem, name=Remark]{remark}

\renewcommand{\P}{\mathbb{P}}
\newcommand{\C}{\mathbb{C}}
\newcommand{\A}{\mathcal{A}}
\newcommand{\Z}{\mathbb{Z}}
\newcommand{\Q}{\mathbb{Q}}
\renewcommand{\O}{\mathcal{O}}
\renewcommand{\S}{\mathcal{S}}

\newcommand{\eps}{\varepsilon}
\newcommand{\w}{\omega}
\newcommand{\xqed}{\hspace*{\fill}\spadesuit}
\DeclareMathOperator{\ext}{Ext^\bullet}
\DeclareMathOperator{\hh}{H^\bullet}
\newcommand{\T}{\widetilde{T}}
\newcommand{\E}{\widetilde{E}}
\newcommand{\F}{\widetilde{F}}
\begin{document}
\title{Towards a full exceptional collection for the adjoint grassmannian of type E6}
\author{Valentin Boboc}
\address{The University of Manchester Alan Turing Building, Oxford Road, Manchester M13 9PL}
\email{valentinboboc@icloud.com}

\begin{abstract}
We construct a Lefschetz exceptional collection of vector bundles in the bounded derived category of coherent sheaves of the adjoint/coadjoint Grassmannian of type $E_6$ of dimension $21$. 
\end{abstract}
\maketitle
\section{Introduction}
There is a long history of studying exceptional collections of derived categories of algebraic varieties. In particular, it is conjectured that the derived category $\mathbf{D}^b(X)$ of any homogeneous space $X = G/P$ has a full exceptional collection. For homogeneous spaces associated to Lie groups of classical type $A_n$, $B_n$, $C_n$, $D_n$, the variety $G/P$ corresponds to Grassmannians of lines, orthogonal lines, isotropic lines, and Lagrangian lines, respectively. Their derived categories were amply studied by various authors. The reader may consult \cite{guseva2020derived}\cite{kapranov1988derived}\cite{kuznetsov2008exceptionaliso}\cite{kuznetsov2016exceptionalisotropic}\cite{kuznetsov2006hyperplane}\cite{samokhin2001derived}  together with the references therein. 

For the exceptional groups $E_n$, $F_4$, the only known cases are the Cayley plane together with its dual \cite{faenzi2015derived}\cite{manivel2009derived}, the adjoint variety of type $F_4$ \cite{smirnov2021derived}, and the coadjoint variety of type $F_4$ obtained as a hyperplane section of the Cayley plane \cite{belmans2021derived}. The known cases have been shown to possess Lefschetz exceptional collections in the sense of \cite{kuznetsov2007homological}. 

In this article we investigate the case of $X = E_6/P_2$, the adjoint (and coadjoint) Grassmannian of type $E_6$. This note focuses on proving the following result.

\begin{mainthm}
Let $X = E_6/P_2$ be the adjoint Grassmannian of type $E_6$. The collection of vector bundles with starting block \[\mathbf{Q}_0 = \langle \O, \T, \S^{w_1}, \S^{2w_1}, \E, \F, \S^{3w_1}, \S^{4w_1}, \S^{5w_1} \rangle\]
supported on the partition $p = [9, 8, 7, 6, 6, 6, 6, 6, 6, 6, 6]$ is a Lefschetz exceptional collection for $\mathbf{D}^b(X)$ of maximal length. 
\end{mainthm}

\noindent \textbf{Organisation.} Section~\ref{sec:prelim} is dedicated to theoretical background. We introduce homogenous varieties and various homological calculations exemplified in the concrete case of $E_6/P_2$. Section~\ref{sec:3} demonstrates some technical calculations. In Section \ref{sec:main} we collect all the vector bundles necessary to build our exceptional collection and prove our main theorem using the tools developed in the first part of the paper. 


\section{Preliminaries}\label{sec:prelim}
We work over the field of complex numbers. Let $\A$ be a $\C$-linear triangulated category. In general, triangulated categories can have rather complicated structures. If a $\C$-linear category $\A$ possesses a full exceptional collection $E_1, E_2, \ldots, E_n$, then there is a unique filtration of $\A$ where each subquotient is a direct sum of shifts of the corresponding $E_i$. A full exceptional collecion is akin to a basis of the triangulated category. 

\begin{definition}
A sequence of full triangulated subcategories $\mathcal{A}_1, \mathcal{A}_2, \ldots, \A_n \subset \A$ is \emph{semiorthogonal} if for all $0\leq i < j\leq n$ and $E \in \A_i$, $F\in \A_j$, we have $\ext(E,F) = 0$. Let $\langle \A_1, \A_2, \ldots, \A_n\rangle$ be the smallest triangulated category containing all $\A_i$. If $\A = \langle \A_1, \A_2, \ldots, \A_n\rangle$, then the subcategories $\A_i$ form a \emph{semiorthogonal decomposition} of $\A$. 
\end{definition}

\begin{definition}
An object $E \in \A$ is \emph{exceptional} if $E$ is simple and has no non-trivial extensions. In symbols: $\ext(E,E) = \C$.  
\end{definition}

\begin{definition}\label{defexceptionalcollection}
A sequence of objects $ E_1, E_2, \ldots, E_n$ is called \emph{exceptional} if the following two conditions are satisfied:
\begin{enumerate}
    \item $E_i$ is an exceptional object for all $1 \leq i \leq n$,
    \item $\ext(E_i, E_j) = 0$ for all $i>j$.
\end{enumerate}

If the smallest subcategory containing the $E_i$'s coincides with $\A$, then we say that the sequence is \emph{full}. 
\end{definition}

\begin{definition}\label{lefschetzdef}
A \emph{Lefschetz collection} $\mathcal{M}$ with respect to some line bundle $\mathcal{L}$ is a collection of vector bundles consisting of a starting block $\mathbf{B}_0 = \langle E_0, E_1, \ldots, E_{p_0} \rangle$ together with a support partition $p = [p_0, p_1, \ldots, p_{l-1}]$, where $p_0 \geq p_1 \geq \ldots \geq p_{l-1} >0$ is a non-increasing sequence of positive integers and the collection $\mathcal{M}$ is organised into $l>0$ distinct blocks of length prescribed by the partition $p$ as such:
\[\underbrace{E_0, E_1, \ldots, E_{p_0}}_{\text{block }\mathbf{B}_0}, \underbrace{E_{0}\otimes \mathcal{L}, E_1 \otimes \mathcal{L},\ldots E_{p_1}\otimes \mathcal{L}}_{\text{block } \mathbf{B}_1},\ldots, \underbrace{E_0\otimes \mathcal{L}^{l-1}, \ldots E_{p_{l-1}}\otimes \mathcal{L}^{l-1}}_{\text{block } \mathbf{B}_{l-1}}.\]

If $p_0 = p_1 = \ldots = p_{l-1}$, then the collection is called \emph{rectangular}. 
\end{definition}

\begin{example} There are many examples of Lefschetz collections in the literature.

\begin{enumerate}
    \item Any exceptional collection is trivially a Lefschetz collection with one block. 
    \item For projective space $\P^n$, the standard Beilinson collection \[\O_{\P^n}, \O_{\P^n}(1), \O_{\P^n}(2), \ldots, \O_{\P^n}(n)\]
    is a Lefschetz collection with respect to $\O_{\P^n}(d)$ with support partition \[p = [\underbrace{d,d,\ldots,d}, r]\] where $n+1 = qd+r$ and $0<r\leq d$. 
    \item For the Grassmannian $\mathrm{Gr}(2,4)$, the collection \[\O, \mathcal{U}^*, S^2\mathcal{U}^*, \O(1), S^2\mathcal{U}^*, \O(2)\] is a Lefschetz collection with respect to $\O(1)$ with support partition \[p = [3,2,1],\]
    where $\mathcal{U}$ is the tautological bundle. The Grassmannian $\mathrm{Gr}(2,4)$ coincides with a quadric and its derived category of coherent sheaves was analysed in \cite{kapranov1988derived}.
    \item The Lefschetz collections for other generalised Grassmannians as cited in the prequel. 
\end{enumerate}
\end{example}
\begin{lemma}\label{lefscetzexceptional}
Let $\mathcal{M}$ be a Lefschetz collection with starting block $\mathbf{B}_0$ and partition $p$. The collection $\mathcal{M}$ is an exceptional collection if and only if the following two conditions are both satisfied:
\begin{enumerate}
    \item the starting block $\mathbf{B}_0$ is an exceptional collection,
    \item $\ext(E_m, E_n(-i)) = 0$ for $1 \leq i \leq l - 1$, $1 \leq m \leq p_{i}$, and $1 \leq n \leq p_0$.  
\end{enumerate}
\end{lemma}
\begin{proof}
A Lefschetz exceptional collection is completely determined by its starting block $\mathbf{B}_0$. One can in fact construct subsequent blocks inductively: $\mathbf{B}_j = {}^\perp \mathbf{B}_0 \cap \mathbf{B}_{j-1}$. It then suffices to verify that all other blocks $\mathbf{B}_j$ are generated by subcollections of $\mathbf{B}_0$. So we only need to check semi-orthogonality between the objects of these collections. The conclusion follows from remarking the following property of extension groups: \[\ext(E_m(i), E_n(j)) = \ext(E_m, E_n(j-i)).\]
\end{proof}
\begin{remark}
Suppose $X$ is Fano and its canonical bundle is $w_X \simeq \O(-r)$, i.e. $r$ is the index of $X$. Then a Lefschetz exceptional collection with respect to $\O(1)$ must have at most $r$ blocks. This is because for an exceptional vector bundle $E$, Serre duality gives \[\ext(E \otimes \O(r), E) \simeq \mathrm{Ext}^{\mathrm{dim}(X) - \bullet}(E, E)^\vee \simeq  \C[-\mathrm{dim}(X)].\]

Thus the maximal length of a full Lefschetz collection is expected to be equal to the rank of the Grothendieck group and the length of the starting block is expected to be no more than the quotient of the rank of the Grothendieck group by the index of $X$. 

$\xqed$
\end{remark}

\subsection{Generalised Grassmannians}
Let $G$ be a connected simple algebraic group. We fix the chain of inclusions \[T \subset B \subset P \subset G\]
where $T$ is a maximal torus, $B$ is the Borel subgroup of $G$, and $P$ is a parabolic subgroup of $G$. Take $P = L P^u$ to be the Levi decomposition of $P$, where $L$ is the semi-simple part, and $P^u$ is the unipotent radical of $P$. In this setting, the quotient $G/P$ is a smooth projective variety, and there is an equivalence of monoidal categories between the category of $G$-equivariant coherent sheaves on $G/P$ and the category of finite dimensional representations of $P$. In symbols: \[\mathrm{Coh}^G(G/P) \simeq \mathrm{Rep}(P).\]

The unipotent radical $P^u$ acts trivially on irreducible $P$-modules, and thus such modules are completely determined by the action of the semi-simple part $L$. Since $L$ is reductive, irreducible representations of $P$ are described by the highest weight theory. 

The inclusion $L \to G$ determines an isomorphism of weight lattices $P_L \simeq P_G$. We denote by $P^+_L$, $P^+_G$ the semigroups of dominant weights of $L$ and $G$, respectively. The inclusion also induces a monomorphism of the corresponding Weyl groups $W_L \to W_G$. Denoting by $\Delta_G$ the root system of $G$ and by $\alpha_i$ the simple roots, then $W_G$ is generated by all the simple reflections $s_{\alpha_i}$ corresponding to the simple roots. For $P$ a maximal parabolic subgroup, the root system of $L$, $\Delta_L$, is given by the simple roots $\alpha_i$ of $G$ with the exception of a prescribed $\alpha_k$ for some $k$ and the Weyl group $W_L$ is generated by the corresponding subset of simple reflections. 

We call $\rho_G$ the sum of fundamental weights of $G$, $w_0$ the longest element of $W_G$, and $w_0^L$ the longest element of $W_L$. They satisfy $(w_0)^2 = (w_0^L)^2 = 1$. Moreover, $w_0$ takes any simple root to the negative of a simple root, and similarly any fundamental weight to the negative of a fundamental weight.

\begin{definition}
A weight $\mu \in P_G$ is \emph{singular} if $\mu$ is orthogonal to some root, or equivalently if $\mu$ lies on a wall of some Weyl chamber of the action of $W$ on $P_G$. A weight $\mu = c_1w_1 + c_2w_2 + c_3w_3 +c_4w_4+c_5w_5+c_6w_6$ is called \emph{dominant} if $c_i \geq 0$ for all $1 \leq i \leq 6$. 
\end{definition}

For a dominant weight $\lambda \in P^+_G$, we denote by $V^\lambda$ the corresponding irreducible representation of $G$. For a dominant weight $\mu \in P^+_L$, we denote by $V_L^\mu$ the irreducible representation of $L$ and by $\S^\mu$ the corresponding $G$-equivariant vector bundle over $G/P$. 

We record the following useful facts.

\begin{lemma}\label{duals} The following hold:
\begin{enumerate}
    \item The dual of a $G$-equivariant vector bundle is given by $(\S^\mu)^\vee \simeq \S^{-w_0^L \cdot \mu}$.
    \item If $V_L^{\lambda} \otimes V_L^\mu = \oplus V_L^\nu$ then $\S^\lambda \otimes \S^\mu = \oplus \S^\nu$. 
\end{enumerate}
\end{lemma}
\begin{proof}
This results from the equivalence $\mathrm{Coh}^G(G/P) \simeq \mathrm{Rep}(P)$ as monoidal categories.
\end{proof}

\begin{theorem}[Borel-Weil-Bott]\label{bwb} Fix a dominant weight $\mu \in P^+_L$ and its associated $G$-equivariant vector bundle $\S^\mu$. If $\mu + \rho_G$ lies on a wall of a Weyl chamber for the action of $W_G$, then we have that \[\mathrm{H}^\bullet(G/P, \S^\mu) = 0.\]

Otherwise, there exists a unique element of the Weyl group $w \in W_G$ such that $w\cdot (\mu +\rho_G)$ is a dominant weight and \[\mathrm{H}^\bullet(G/P, \S^\mu) = V^{w\cdot (\mu+\rho_G) - \rho_G}[-l(w)],\]
where $l : W_G \to \Z$ is the length function. 
\end{theorem}

\subsection{The case X = E6/P2}
In particular, we work with the exceptional group $G = E_6$ and the maximal parabolic subgroup $P = P_2$ which corresponds to the second vertex of the Dynkin diagram
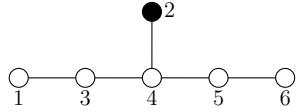
\[\dynkin[labels={1,2,3,4,5,6}, scale=2.5] E{o*oooo}\]
where we use the Bourbaki ordering on the vertices. Then the homogeneous space
\[X = G/P = E_6/P_2\]
is a smooth projective Fano variety of dimension $21$. 

We can realise the weight lattice $P_G$ of $G = E_6$ inside $\Q^8$ in the following manner. Let $\eps_1, \eps_2, \ldots, \eps_8$ be a basis of $\Q^8$. Then the simple roots of $G$ are given by 
\begin{align*}
    &\alpha_1 = \frac{1}{2}(\eps_1 + \eps_8) - \frac{1}{2}(\eps_2 +\eps_3+\eps_4 +\eps_5+\eps_6+\eps_7),\\
    &\alpha_2 = \eps_1 + \eps_2, \quad \quad \alpha_3 = \eps_2 - \eps_1,\\
    &\alpha_4 = \eps_3 - \eps_2, \quad \quad \alpha_5 = \eps_4 - \eps_3,\\
    &\alpha_6 = \eps_5 - \eps_4.
\end{align*}
The space spanned by $\alpha_1, \ldots, \alpha_6$ is orthogonal to the plane generated by $\eps_7 + \eps_8$ and $\eps_6 + \eps_7 + 2\eps_8$. These six simple roots generate the root system $\Delta_G$. 

The fundamental weights can be written in terms of the simple roots as such
\begin{align*}
    &\omega_1  = \frac{1}{3}(4\alpha_1 +3\alpha_2 +5\alpha_3+6\alpha_4+4\alpha_5+2\alpha_6,\\
    &\omega_2 = \alpha_1+2\alpha_2 +2\alpha_3 + 3\alpha_4 +2\alpha_5 +\alpha_6,\\
    &\omega_3 = \frac{1}{3}(5\alpha_1 + 6\alpha_2 +10\alpha_3+12\alpha_4 +8\alpha_5 +4\alpha_6),\\
    &\omega_4 = 2\alpha_1 + 3\alpha_2+4\alpha_3 + 6\alpha_4 + 4\alpha_5 + 2\alpha_6,\\
    &\omega_5 = \frac{1}{3}(4\alpha_1 + 6\alpha_2 + 8\alpha_3 + 12\alpha_4 + 10\alpha_5 + 5\alpha_6),\\
    &\omega_6 = \frac{1}{3}(2\alpha_1 + 3\alpha_2 + 4\alpha_3 + 6\alpha_4 + 5\alpha_5 + 4\alpha_6).
\end{align*}

The sum of fundamental weights is 
\[\rho_G = 8\alpha_1 + 11\alpha_2 + 15\alpha_3 +21\alpha_4 +15\alpha_5 + 8\alpha_6.\]

The longest elements of the Weyl groups $W_G$ and $W_L$ can be expressed as
\begin{align*}
    w_0&= (-1,-1,-1,-1,-1,-1),  \quad \quad \quad w_0^L = (-1,10,-1,-1,-1,-1).
\end{align*}

The Picard group of $X$ is generated by a single line bundle: $\mathrm{Pic}\ X = \Z \langle \O(1) \rangle$, where we employ the following notation for $i \in \Z$
\begin{align*}
    \O(1) = \S^{\omega_2}, \quad \quad \O(i) = \S^{i\omega_2}, \quad \quad \S^{\mu +iw_2} = \S^\mu(i).
\end{align*}

The following lemma tells us how to compute the canonical class of homogeneous spaces.
\begin{lemma}\label{canonical}
Let $X = G/P$ be a homogeneous space as in our setting. Then the canonical bundle is given by \[\omega_X = \S^{w_0^Lw_0\rho_G - \rho_G}.\] 
Equivalently, we have $w_X = \S^{-rw_2} = \O(-r)$ where $r = \langle \rho_G - w_0^Lw_0\rho_G, \alpha_k\rangle / \langle w_k, \alpha_k\rangle$, where $\langle -, - \rangle$ is the standard Euclidean product on the weight lattice $P_G = P_L$ and $k$ is the vertex of the Dynkin diagram of $G$ corresponding to the chosen maximal parabolic subgroup $P$. 
\end{lemma}
\begin{lemma}
For $X = G/P$ where $G = E_6$ and $P = P_2$, the canonical bundle is \[\omega_X = \O(-11).\]
\end{lemma}
\begin{proof}
Take $k=2$ in the previous lemma. We have $\langle w_2, \alpha_2 \rangle = \frac{1}{2}$ and $\rho_G - w_0^Lw_0\rho_G = (0,11,0,0,0,0)$. This immediately gives $r = 11$. 
\end{proof}

\subsection{Tensor products}
We next discuss how to compute tensor products of vector bundles over $X$. The discussion here is a summary of the methods employed in \cite{belmans2019hochschild}\cite{smirnov2021derived}. 

The Levi subgroup $L \subset P$ decomposes as the product $L \simeq L' \times Z(L)$, where $L' \subset L$ is the derived group and $Z(L)\subset L$ is the centre. The derived group $L'$ is a connected semi-simple algebraic group and its Dynkin diagram is obtained by removing the vertex $\alpha_k$ in the Dynkin diagram of $G$ which corresponds to the chosen maximal parabolic subgroup $P$. 

In our specific case, $\alpha_k = \alpha_2$ and the Levi subgroup decomposes as $L \simeq A_5 \times \C^*$.

The inclusions $L' \hookrightarrow L$ and $Z(L) \hookrightarrow L$ induce the following morphisms on weight lattices
\begin{align*}
   f_{L'}: P_L &\rightarrow P_{L'} \quad \quad \quad \quad \quad \quad \quad \ \ f_{Z(L)}: P_L \rightarrow P_{Z(L)}\\
               w_i &\mapsto w_i' \text{ if } i\neq k  \quad \quad \quad \mu = \sum_i c_i w_i \mapsto r_\mu w_k\\
               w_k &\mapsto 0 \quad \quad \quad \quad \quad \quad \quad \quad \text{where } r_\mu = \frac{\langle \mu, w_k \rangle}{\langle w_k, w_k \rangle}.
\end{align*}
We can also define a lifting map $P_{L'} \to P_L$ which sends $\w'_i \mapsto w_i$ for all $i\neq k$. 

Every representation of $L$ can be expressed as a suitable pair of representations of $L'$ and $Z(L)$. For a highest weight, irreducible representation $V_L^\mu$, the restrictions $\mathrm{Res}^L_{L'}(V_L^\mu)$ and $\mathrm{Res}^L_{Z(L)}(V_L^\mu)$ are both irreducible. Moreover, we know that $\mathrm{Res}^L_{L'}(V_L^\mu) = V_{L'}^{\mu'}$, where $\mu' = f_{L'}(\mu)$ and $\mathrm{Res}^L_{Z(L)}(V_L^\mu)$ is given by a character of the torus $Z(L) \simeq \C^*$. 

The following lemma explains how to exploit this simplification to compute tensor products of vector bundles over $X$. 

\begin{lemma}\label{tensorproductlemma}
Let $V_L^\mu, V_L^\nu$ be highest weight irreducible representations of $L$ and let $V_{L'}^{\mu'}, V_{L'}^{\nu'}$ be the restrictions to the derived group $L'$. In that case, if \[V_{L'}^{\mu'} \otimes V_{L'}^{\nu'} = \bigoplus_{\sigma'}(V_{L'}^{\sigma'})^{m(\mu', \nu', \sigma')}\]
then we have 
\[\S^\mu \otimes \S^\nu = \bigoplus_{\sigma'}(\S^{\sigma + (r_\mu +r_\nu - r_\sigma)w_k})^{m(\mu', \nu',\sigma')},\]
where $\sigma$ is the lift of $\sigma'$ in $P_L$, $r_\mu,r_\nu,r_\sigma$ are the constants defined by the map $f_{Z(L)}$, and $m(\mu', \nu',\sigma')$ is the multiplicity of the module in the tensor product decomposition. 
\end{lemma}
\begin{proof}
The decomposition of the tensor product of representations of $L'$ determines a decomposition of the tensor product of the corresponding representations of $L$. By the equivalence of categories stated earlier, these irreducible $P$-modules correspond to $G$-equivariant coherent sheaves.
\end{proof}

In our case, such calculations reduce to decomposing tensor products of representations of $A_5$. We have the following map at the level of Dynkin diagrams
\begin{align*}
    E_6 :\ \ \dynkin[labels={w_1, w_2, w_3, w_4, w_5, w_6}, scale=2.5]E{o*oooo} \quad \quad  \longrightarrow \quad \quad A_5: \ \ \dynkin[labels={\pi_1,\pi_2,\pi_3,\pi_4,\pi_5},scale=2.5]A{ooooo}
\end{align*}

The morphism of weight lattices $f_{L'} : P_L \to P_{L'}$ is given by 
\begin{align*}
    &w_1 \mapsto w_1' = \pi_1, \quad \quad w_3 \mapsto w_3' = \pi_2,\\
    &w_4 \mapsto w_4' = \pi_3, \quad \quad w_5 \mapsto w_5' = \pi_4,\\
    &w_6 \mapsto w_6' = \pi_5.
\end{align*}

The constants that define the restriction to the torus $f_{Z(L)}$ are
\[r_{w_1} = \frac{1}{2}, \quad r_{w_2} = 1, \quad r_{w_3} = 1, \quad r_{w_4} = \frac{3}{2}, \quad r_{w_5}=1, \quad r_{w_6} = \frac{1}{2}.\]
We can then use Lemma \ref{tensorproducts} together with the tables from \cite[p.300]{onishchik2012lie} to calculate all tensor product decompositions of interest. We give an explicit example. 
\begin{example}\label{tensorexample}
Take $X = E_6/P_2$ with the conventions above. We compute $\S^{w_6}(-1)\otimes \S^{w_1}$. The weights, their restrictions, and the corresponding constants are
\begin{align*}
    \mu &= w_6 - w_2  \quad \quad \mu' = w_6' \quad \quad \quad r_{\mu} = -\frac{1}{2}\\
    \nu &= w_1  \quad \quad \quad \quad \ \ \nu' = w_1' \quad \quad \quad  r_\nu = \frac{1}{2}.
\end{align*}

We first need to compute the decomposition of $V_{A_5}^{w_6'}\otimes V_{A_5}^{w_1'}$. We use our rules for restricting weights to $A_5$ together with the tables and the conventions from \cite{onishchik2012lie}. For the intermediate computations in $A_5$, $R(\pi_i)$ is the representation corresponding to the $i$-th vertex of the Dynkin diagram, $R(0) = 1$ is the unit representation, and meaningless symbols such as $R(-\pi_3+\pi_7)$ are disregarded.
\begin{align*}
    V_{A_5}^{w_6'}\otimes V_{A_5}^{w_1'} &= R(\pi_5) \cdot R(\pi_1)\\
    &= \sum_{i\geq 0} R(\pi_{5+i} + \pi_{1-i})\\
    &= R(\pi_5 + \pi_1) + R(\pi_6 + \pi_0) + R(\pi_7 + \pi_{-1}) + \ldots\\
    &= R(\pi_5 + \pi_1) + R(0) + 0\\
    &= 1 + R(\pi_1 + \pi_5)\\
    &= 1 \oplus V_{A_5}^{w_1' +w_6'}.
\end{align*}

The set of coefficients of the decomposition is $\sigma' \in \{0, w_1'+w_6'\}$ which lifts to $\sigma \in \{0, w_1+w_6\}$ with coefficients $r_0 = 0$, and $r_{w_1+w_6} = r_{w_1} + r_{w_6} = \frac{1}{2} + \frac{1}{2} = 1$. Lemma \ref{tensorproducts} then gives
\begin{align*}
    \S^{w_6-w_2}\otimes \S^{w_1} &= \O^{0+ 0 \cdot w_2} \oplus \S^{w_1 +w_6 + (\frac{1}{2}-\frac{1}{2}-1)w_2}\\
    &= \O \oplus \S^{w_1+w_6-w_2}\\
    &= \O \oplus \S^{w_1+w_6}(-1).
\end{align*}
$\xqed$
\end{example}
\section{Some calculations}\label{sec:3}
\begin{definition}
We say a line bundle $L \in \mathrm{Pic}\ X$ is \emph{acyclic} if $\mathrm{H}^\bullet(X, L) = 0$. A vector bundle $E$ over $X$ is \emph{exceptional} if \[\ext(E,E) = \C[0].\]
\end{definition}

\begin{lemma}\label{acyclichomogeneous}
Over a homogeneous space $X = G/P$ with $P$ maximal, the vector bundle $\S^\mu$ is acyclic if and only if $\mu + \rho_G$ is a singular weight. 
\end{lemma}
\begin{remark}\label{remarkacyclic}
Since $E_6$ is simply laced, verifying that a bundle $\S^\mu$ is acyclic can theoretically be done by hand following this algorithm. If $\mu$ is dominant, stop. Otherwise, one of the coefficients of $\mu$ at a funda,mental weight, say $w_j$, must be negative. Apply the reflection $s_{\alpha_j} \in W_G$ to $\mu$. If the resulting weight is dominant, stop. If the resulting weight has non-negative coefficients and at least one of the coefficients is zero (i.e. the resulting weight is singular), then the bundle is acyclic and the algorithm terminates. If negative coefficients still exist, apply reflections until either a singular weight or a dominant weight is obtained. In the case we obtain a dominant weight, the number of steps it took to reach the dominant weight is equal to the length $l(w)$ of the unique Weyl group element specified in Theorem \ref{bwb}. 
\end{remark}
\begin{example}\label{acyclicexample}
For $X = E_6/P_2$, we show that $\S^\mu = \S^{w_1+w_6}(-4)$ is acyclic. Equivalently, by Lemma \ref{acyclichomogeneous}, we need to show that \[\mu+\rho_G = w_1+w_6-4w_2+\rho_G = (2,-3,1,1,1,2) \in P_G = P_L\]
is a singular weight. We can describe this vector bundle using the Dynkin diagram of $E_6$ by labelling every vertex with the coefficient attached to the fundamental weight of that vertex. Here we have:
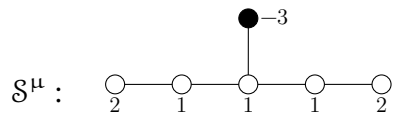
\[\S^{\mu}: \quad \dynkin[labels={2,-3,1,1,1,2},scale=2.5]E{o*oooo}\]
We now follow the algorithm from the previous remark to obtain a singular weight. 
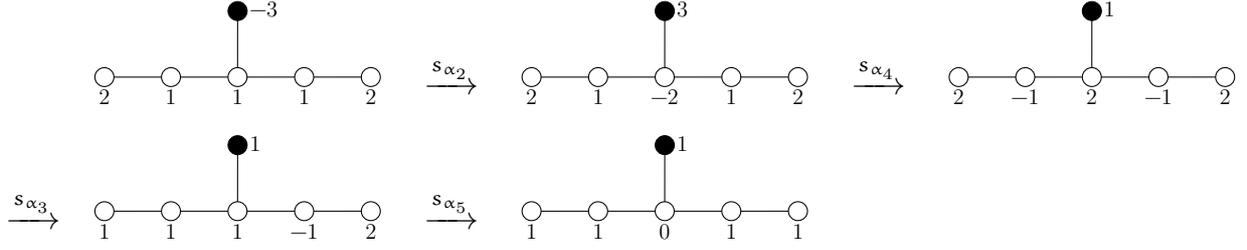
\begin{align*}
   &\dynkin[labels={2,-3,1,1,1,2},scale=2.5]E{o*oooo} \quad \xrightarrow[]{s_{\alpha_2}} \quad \dynkin[labels={2,3,1,-2,1,2},scale=2.5]E{o*oooo} \quad \xrightarrow[]{s_{\alpha_4}} \quad  \dynkin[labels={2,1,-1,2,-1,2},scale=2.5]E{o*oooo}\\
   \xrightarrow[]{s_{\alpha_3}}\quad  &\dynkin[labels={1,1,1,1,-1,2},scale=2.5]E{o*oooo} \quad \xrightarrow[]{s_{\alpha_5}} \quad \dynkin[labels={1,1,1,0,1,1},scale=2.5]E{o*oooo}
\end{align*}

The final diagram has a null weight assigned to the $w_4$ vertex. This shows that the weight $\mu+\rho_G$ is orthogonal to a root, i.e. singular. In particular, the Borel-Weil-Bott Theorem tells us that that all cohomology groups vanish. In symbols: $\mathrm{H}^\bullet(X, \S^\mu) = 0$. 

Similarly we can show that $\S^\nu = \S^{w_1+w_3+w_5}(-3)$ is not acyclic. This amounts to checking whether the weight \[\nu +\rho_G = w_1+w_3+w_5-3w_2+\rho_g = (2,-2,2,1,2,1) \in P_G = P_L\]
is singular or not. We proceed with the algorithm.
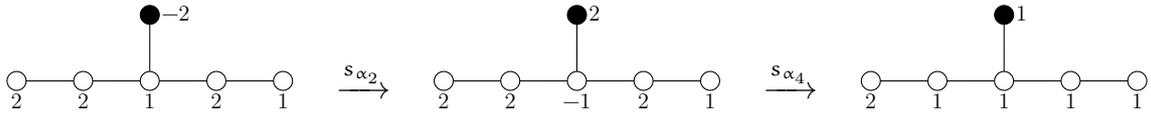
\begin{align*}
    &\dynkin[labels={2,-2,2,1,2,1},scale=2.5]E{o*oooo} \quad \xrightarrow[]{s_{\alpha_2}} \quad  \dynkin[labels={2,2,2,-1,2,1},scale=2.5]E{o*oooo} \quad \xrightarrow[]{s_{\alpha_4}} \quad  \dynkin[labels={2,1,1,1,1,1},scale=2.5]E{o*oooo}
\end{align*}

The final Dynkin diagram has strictly positive weights on all vertices. The algorithm stops and we conclude that that the weight is dominant. The Weyl group element $w \in W_G$ which makes $\nu+\rho_G$ dominant is $w = s_{\alpha_2}s_{\alpha_4}$ and it has length $l(w) = 2$. Moreover, we have that $w(\nu +\rho_G)-\rho_G = w_1$ and we can apply the Borel-Weil-Bott Theorem to determine the cohomology groups \[\mathrm{H}^\bullet(X, \S^\nu) = V_G^{w_1}[-2].\]
$\xqed$
\end{example}
\begin{lemma}
Let $\S^{\mu}(-m)$ be a vector bundle over $X$ with $m>0$ and \[\mu = a_1w_1+a_3w_3+a_4w_4+a_5w_5+a_6w_6.\] If $a_i > 0$ for all $i$ and $a_4 = m$, then $\S^{\mu}(-m)$ is acyclic. 
\end{lemma}
\begin{proof}
Immediate from Remark \ref{remarkacyclic} and applying the reflection $s_{\alpha_2}$. 
\end{proof}
\begin{lemma}\label{cohomology}
We note the following cohomology computations for $X = E_6/P_2$.
\begin{align*}
     \mathrm{H}^\bullet(X, \O) = \C[0],\\
  \hh(X, \S^{w_4}(-2)) = \C[-1]. 
\end{align*}

\end{lemma}
\begin{proof}
Statement (1) follows from Kodaira vanishing as $X$ is smooth, connected, and Fano. Statement (2) follows from the Borel-Weil-Bott Theorem. It is immediate to see that for $\mu = w_4-2w_2$, we have $s_{\alpha_2}(\mu+\rho) = \rho$.
\end{proof}

\begin{lemma}\label{tensorproducts}
We note the following tensor product computations.
\begin{align*}
    \S^{uw_6}\otimes \S^{tw_1} &= \bigoplus_{j=0}^{\mathrm{min}(u,t)} \S^{(u-j)w_6+(t-j)w_1}(j),\\
    \S^{w_4}\otimes \S^{w_4} &= \O(3) \oplus \S^{w_1+w_6}(2) \oplus \S^{w_3+w_5}(1) \oplus \S^{2w_4},\\
    \S^{w_4}\otimes \S^{tw_1} &= \S^{w_4+tw_1} \oplus \S^{w_5+(t-1)w_1}(1),\\
    \S^{w_4}\otimes \S^{tw_6} &= \S^{w_4+tw_6} \oplus \S^{w_3+(t-1)w_6}(1),\\
    \S^{w_4}\otimes \S^{w_3} &= ,\S^{w_3+w_4} \oplus \S^{w_1+w_5}(1) \oplus \S^{w_6}(2),\\
    \S^{w_4}\otimes \S^{w_5} &= \S^{w_4+w_5} \oplus \S^{w_3+w_6}(1) \oplus \S^{w_1}(2),\\
    \S^{w_4}\otimes \S^{w_1+w_6} &= \S^{w_1+w_4+w_6} \oplus \S^{w_5+w_6}(1) \oplus \S^{w_4}(1) \oplus \S^{w_1+w_6}(1),\\
    \S^{tw_6}\otimes \S^{w_1+w_6} &= \S^{w_1+(t+1)w_6} \oplus \S^{w_1+w_5+(t-1)w_6} \oplus \S^{tw_6}(1) \oplus \S^{w_5+(t-2)w_6}(1),\\
    \S^{tw_6}\otimes \S^{w_3} &= \S^{w_3+tw_6}\oplus \S^{w_1+(t-1)w_6}(1),\\
    \S^{w_5}\otimes \S^{tw_1} &= \S^{w_5+tw_1} \oplus \S^{w_6+(t-1)w_1}(1),\\
    \S^{w_5}\otimes \S^{w_1+w_6} &= \S^{w_1+w_5+w_6} \oplus \S^{w_5}(1) \oplus \S^{2w_6}(1)\oplus \S^{w_1+w_4},\\
    \S^{w_6} \otimes \S^{w_3} &= \S^{w_3+w_6} \oplus \S^{w_1}(1).
\end{align*}
\end{lemma}
\begin{proof}
Calculations in the style of Example \ref{tensorexample} can be done with the help of Lemma \ref{tensorproductlemma} together with Table $5$ from \cite[p.300]{onishchik2012lie}. We used \texttt{GAP} \cite{gap2021} for decomposing certain representations of $A_5$. 
\end{proof}
\begin{lemma}\label{acyclic}
The following vector bundles are acylic.
\begin{align*}
   \O(-i) &\text{ for } 1 \leq i \leq 10,\\
   \S^{tw_j}(-i) &\text{ for } 1 \leq i \leq 10, 1 \leq t \leq 5, j \in \{1,3,4,5,6\},\\
   \S^{w_1+bw_6}(-i-b) &\text{ for } 1 \leq i \leq 10, 1\leq b \leq 6,\\
    \S^{aw_1+w_6}(-i-b) &\text{ for } 1 \leq i \leq 10, 1\leq a \leq 6,\\
   \S^{w_1+w_4}(-i) &\text{ for } 3 \leq i \leq 12, \\
   \S^{w_1+w_5}(-i) &\text{ for } 2 \leq i \leq 11, \\
   \S^{w_4+w_6}(-i) &\text{ for } 1 \leq i \leq 10,\\
   \S^{w_5+w_6}(-i) &\text{ for } 3 \leq i \leq 14,\\
   \S^{w_1+w_5+bw_6}(-i) &\text{, for } 2 \leq i \leq 11, 2\leq b \leq 4,\\
   \S^{aw_6+(a-1)w_1}(-i-a) &\text{ for } i \in \{0,1\}, a \in \{2,3,4,5\},\\
   \S^{(a-1)w_6+aw_1}(-i-a+1) &\text{ for } i \in \{1,2\}, a \in \{2,3,4,5\}.
\end{align*}
\end{lemma}
\begin{proof}
These calculations can be done in the style of Example \ref{acyclicexample} using Lemma \ref{acyclichomogeneous}. Alternatively, it is straightforward to implement the algorithm from Remark \ref{remarkacyclic} in \texttt{GAP} \cite{gap2021}.
\end{proof}
\begin{lemma}\label{exceptional}
The following vector bundles are exceptional: \[\O \text{ and  } \S^{tw_1} \text{ for  } 1 \leq t \leq 5.\]
\end{lemma}
\begin{proof}
We know that $\O$ is exceptional by Lemma \ref{cohomology}. We give a detailed argument for $\S^{w_1}$. By definition, $\S^{w_1}$ is exceptional if and only if $\ext(\S^{w_1}, \S^{w_1}) = \C[0]$. To calculate the $\mathrm{Ext}$-groups, we use our conventions about duals, the properties of $\mathrm{Ext}$ groups, and Lemma \ref{tensorproductlemma} (or alternatively, Example \ref{tensorexample}) to decompose relevant tensor products of $G$-equivariant bundles.
\begin{align*}
    \ext(\S^{w_1}, \S^{w_1}) &= \ext(\O, (\S^{w_1})^\vee \otimes \S^{w_1}) \\
    &= \ext(\O, \S^{-w_0^Lw_1} \otimes \S^{w_1})\\
    &= \ext(\O, \S^{w_6}(-1) \otimes \S^{w_1}\\
    &= \ext(\O, \O \oplus \S^{w_1+w_6}(-1))\\
    &= \hh(X, \O \oplus \S^{w_1+w_6}(-1))\\
    &= \hh(X, \O) \oplus \hh(X, \S^{w_1+w_6}(-1)).
\end{align*}

Lemma \ref{cohomology} tells us that $\hh(X,\O) = \C[0]$. It remains to show that $\hh(X, \S^{w_1+w_6}(-1))$ vanishes, or equivalently that $\S^{w_1+w_6}(-1)$ is acyclic. But this is covered in Lemma \ref{acyclic}. Hence $\S^{w_1}$ is exceptional. 

The proof for the remaining vector bundles is similar. 
\end{proof}

\subsection{The tangent bundle}
We compute explicitly the tangent and cotangent bundles of $X = E_6/P_2$. We can obtain a nice formula mainly due to the fact that $w_2$ is an adjoint weight. The tangent bundle $T_X$ itself is not an exceptional object, but certain non-trivial extensions of it are exceptional and we use these to build the exceptional collection $\mathcal{Q}$. 
\begin{lemma}\label{tangentbundle}
For $X = E_6/P_2$, the following hold:
\begin{enumerate}
    \item The tangent bundle $T_X$ is given by the extension
    \[0 \to \S^{w_4}(-1) \to T_X \to \O(1) \to 0.\]
    \item There exists a non-trivial extension $\T$ of the form
    \begin{align*}
        0 \to \O \to \widetilde{T} \to T_X \to 0.
    \end{align*}
\end{enumerate}
\end{lemma}
\begin{proof}
For part (1), the tangent bundle of $G/P$ is given by the Lie algebra quotient \[T_X = \mathfrak{g}/\mathfrak{p} = \mathfrak{n}^\vee,\] where $\mathfrak{n}$ is the Lie nilradical of the parabolic subalgebra $\mathfrak{p}$. As is the case for Lie groups, the parabolic subalgebra can be expressed as $\mathfrak{p} = \mathfrak{l} \oplus \mathfrak{n}$, a Levi subalgebra together with the nilradical. By construction, the nilradical then decomposes into the direct sum \[\mathfrak{n}^\vee = \bigoplus_{\beta \in \Psi}\mathfrak{g}_\beta,\]
where $\Psi$ is the set of all positive roots of $\mathfrak{g}$, which are non-parabolic, i.e. when expressed as a linear combination of simple roots, $\beta$ has a positive coefficient at $\alpha_2$. In symbols: \[\langle \beta, \alpha_2^\vee \rangle \geq 1.\]

Using an algebra package, we can list all the positive roots of $\mathfrak{e}_6$ then search for roots which satisfy our conditions. There are $21$ positive non-parabolic roots, of which only $2$ are dominant: $w_2$ and $w_4 -w_2$. By highest weight theory, we have that \[\mathfrak{n}^\vee = \mathfrak{g}_{w_2} \oplus \mathfrak{g}_{w_4 - w_2} = \O(1) \oplus \S^{w_4}(-1).\]

Since $w_2$ is an adjoint weight, the lower central series of the nilradical is given by
\[0 \to [\mathfrak{n}, \mathfrak{n}] \to \mathfrak{n} \to \mathfrak{n}/[\mathfrak{n}, \mathfrak{n}]\to 0.\]

In this case, the commutator is one-dimensional and it is isomorphic to the irreducible representation of the negative of the highest weight, i.e. $[\mathfrak{n}, \mathfrak{n}] = \mathfrak{g}_{-w_2} = \O(-1)$. Consequently, the abelianisation of the nilradical is $\mathfrak{n}/[\mathfrak{n}, \mathfrak{n}] = \mathfrak{g}_{w_4 - 2w_2} = \S^{w_4}(-2)$. 

Dualising the lower central series, we obtain the desired short exact sequence:
\[0 \to \S^{w_4}(-1) \to T_X \to \O(1) \to 0.\]

Moreover, we can show that this makes $T_X$ the unique nontrivial extension of this form. We can see this by computing the extension groups using the Borel-Weil-Bott Theorem:
\[\ext(\O(1), \S^{w_4}(-1)) = \ext(\O, \S^{w_4}(-2)) = \hh(X, \S^{w_4}(-2)) = \C[-1].\]

Part (2) also follows from an application of Borel-Weil-Bott as in Lemma \ref{cohomology}. For extensions of $T_X$ by $\O$. We see that
\[\ext(T_X, \O) = \ext(\O, T_X^\vee) = \hh(X, \Omega_X) = \hh(X, \O(1) \oplus \S^{w_4}(-2)) = \C[-1],\]
since $\O(1)$ is acyclic and we already computed the cohomology of $\S^{w_4}(-2)$ in the previous step. Hence there exists a unique extension of $T_X$ by $\O$ and we denote this by $\widetilde{T}_X$.

\end{proof}
\begin{lemma}\label{moreextensions}
There exist non-trivial extensions:
\begin{align*}
    0 \to \S^{w_1}(1) \to \E \to \O(1) \to 0,\\
    0 \to \S^{w_6}(1) \to \F \to \O(1) \to 0.
\end{align*}
\begin{proof}
Immediate by the Borel-Weil-Bott Theorem.
\end{proof}
\end{lemma}
\begin{definition}
Let $E$ be a $G$-equivariant vector bundle over $X = G/P$. Its semi-simplification $\beta(E)$ is the vector bundle whose associated representation of the parabolic subgroup $P$ of $G$ is the direct sum of all semi-simple factors of the representation corresponding to $E$.
\end{definition}
\begin{lemma}\label{ss}
We have the following semi-simplifications
\begin{align*}
    \beta(\T) &= \O \oplus \S^{w_4}(-1) \oplus \O(1),\\
    \beta(\E) &= \O(1) \oplus \S^{w_1}(1),\\
    \beta(\F) &= \O(1) \oplus \S^{w_6}(1).
\end{align*}
\end{lemma}
\begin{lemma}\label{dualextension}
There is a $G$-equivariant isomorphism: $\widetilde{T} = (\widetilde{T}_X)^\vee (1)$.
\end{lemma}
\begin{proof}
This follows by studying the convolutions of the sequences from Lemma \ref{tangentbundle}.
\end{proof}
\begin{lemma}\label{mutationlemma}
We have the following mutations of complexes in $\mathbf{D}^b(X)$:
\begin{align*}
    &\mathbf{R}_{\langle \O \rangle}(T_X) = \T,\\
    &\mathbf{L}_{\langle \O(1) \rangle}(\S^{w_1}(1)) = \E,\\
    &\mathbf{L}_{\langle \O(1) \rangle}(\S^{w_6}(1)) = \F.
\end{align*}
\end{lemma}
\begin{proof}
The exact triangles we need are the short exact sequences from Lemmas \ref{moreextensions} and \ref{tangentbundle}.
\end{proof}
\begin{lemma}\label{cohomologyextension}
The following is an exact $3$-term complex
\begin{align*}
    \T(-1) \to (V^{w_2} \oplus \C) \otimes \O \to \T.
\end{align*}
whose cohomology is given by the vector bundle
\begin{align*}
    \S^{w_1+w_6}(-1).
\end{align*}
\end{lemma}
\begin{proof}
From the exact sequences in Lemma \ref{tangentbundle} and the Borel-Weil-Bott Theorem, we can compute the global sections
\begin{align*}
    &\mathrm{H}^0(X, T_X) = V^{w_2},\\
    &\mathrm{H}^0(X, \T) = V^{w_2}\oplus \C.
\end{align*}

Since $T_X$ is globally generated, it follows that $\T$ is also globally generated. Moreover, the surjection $V^{w_2} \otimes \O \to \T$ induces a surjection $(V^{w_2}\oplus \C) \otimes \O \to \T$. Using Lemma \ref{dualextension}, we can dualise this to obtain an injection $\T \to (V^{w_2}\oplus \C) \otimes \O$. Thus the composition \[\widehat{T}_X(-1) \to (V^{w_2} \oplus \C) \otimes \O \to \widehat{T}_X\] vanishes, i.e. this is indeed a complex. 

Concerning the cohomology of the complex, the $G$-module $V^{w_2}$ is $78$-dimensional, and the extension $\T$ has rank 
\[\mathrm{rk}(\T) = \mathrm{rk}(\O) + \mathrm{dim}(V_L^{w_4}) + \mathrm{rk}(\O(1)) = 1+20+1 = 22\]
by highest weight theory. Taking alternating sums of the dimensions of the entries, we see that the cohomology of the complex must be a vector bundle of rank $(78+1) - 22 - 22 = 35$. It is immediate to verify that $\S^{w_1+w_6}(-1)$ has rank $35$.  
\end{proof}

\section{The exceptional collection}\label{sec:main}
\subsection{The building blocks}

\begin{lemma}
The rank of the Grothendieck group of $X = E_6/P_2$ is $72$. 
\end{lemma}
\begin{proof}
This number is equal to the rank of the homology of $E_6/P_2$. Using the Bruhat cell decomposition, it is known that the rank is equal to the quotient of the cardinalities of Weyl groups. Using a computer algebra package, we observe that \[\frac{\#W_G}{\#W_L} = \frac{\#W_{E_6}}{\#W_{L_2}}=72.\]
\end{proof}
\begin{lemma}\label{exceptionalblocks}
The sequences of vector bundles
\begin{align*}
    &\mathcal{A}= \langle \T, \T(1), \T(2), \ldots, \T(10)\rangle,\\
    &\mathcal{B}= \langle \E, \E(1), \E(2), \ldots, \E(10)\rangle,\\
    &\mathcal{C}= \langle \F, \F(1), \F(2), \ldots, \F(10)\rangle\\
\end{align*}
are exceptional collections for $\mathbf{D}^b(X)$.
\end{lemma}
\begin{proof}
For the first collection, we must verify that $\ext(\T(i), \T) = 0 \text{ for } 1 \leq i \leq 10$ and that $\T$ is an exceptional vector bundle, i.e. $\ext(\T, \T) = \C[0]$.

We prove the first vanishing in the cases $2 \leq i \leq 9$. Note that $\T$ is an object of the category spanned by $\langle \O, \O(1), \S^{w_4}(-1)\rangle$ by Lemma \ref{tangentbundle}. Then it suffices to show that 
\begin{align*}
    &\ext(\O(i), \O) = \hh(X, \O(-i)) = 0 \text{ for } 1\leq i \leq 10,\\
    &\ext(\S^{w_4}(i-1), \O) = \hh(X, \S^{w_4}(-i-4)) = 0 \text{ for } 1 \leq i \leq 9,\\
    &\ext(\O(i), \S^{w_4}(-1)) = \hh(X, \S^{w_4}(-i-1)) = 0 \text{ for } 2 \leq i \leq 10,\\
    &\ext(\S^{w_4}(i-1), \S^{w_4}(-1)) = 0 \text{ for } 2 \leq i \leq 9.
\end{align*}
The first three vanishings follow from Lemma \ref{acyclic}. By Lemma \ref{tensorproducts}, we can rewrite the fourth extension group as 
\begin{gather*} \ext(\S^{w_4}(i-1), \S^{w_4}(-1)) = \hh(X, \S^{w_4}(-i-2) \otimes \S^{w_4}(-1))\\  = \hh(X, \O(-i) \oplus \S^{2w_4}(-i-3) \oplus \S^{w_3+w_5}(-i-2) \oplus \S^{w_1 + w_6}(-i-1)).
\end{gather*}
Lemma \ref{acyclic} once again shows that all the featured vector bundles are acyclic in the required range. Thus the fourth extension group also vanishes. 

The cases $i =1$ and $i = 10$ are isomorphic by Serre duality. We only need to check the case when $i =1$. Using the complex from Lemma \ref{cohomologyextension}, we need to check that
\begin{align*}
    &\ext(\T(1), \O) = 0,\\
    &\ext(\T(1), \T(-1)) = 0,\\
    &\ext(\T(1), \S^{w_1+w_6}(-1)) = 0. 
\end{align*}

The middle vanishing holds by what we just proved in the case $i =2$. To see the first and the third vanishing, we replace $\T$ with its semi-simplification $\O \oplus \S^{w_4}(-1) \oplus \O(1)$. The vanishing then follows immediately from Lemma \ref{acyclic}. 

To prove that the extension $\T$ is an exceptional vector bundle, we can employ the complex from Lemma \ref{cohomologyextension} to rewrite the condition $\ext(\T, \T) = \C[0]$ as
\begin{align*}
    &\ext(\T, \O) = 0,\\
    &\ext(\T, \T(-1)) = 0,\\
    &\ext(\T, \S^{w_1+w_6}(-1)) = \C[-1].
\end{align*}

The middle vanishing holds by what we just proved above in the case $i =1$. The short exact sequences from Lemma \ref{tangentbundle} together with the duality from Lemma \ref{dualextension} show that the extension $\T$ is in the left orthogonal of the category spanned by $\langle \O \rangle$. This gives the vanishing of the first extension group. Concerning the third extension group, the relation is equivalent to 
\begin{align*}
    &\ext(\O, \S^{w_1+w_6}(-1)) = 0,\\
    &\ext(\O(1), \S^{w_1+w_6}(-1)) = 0,\\
    &\ext(\S^{w_4}(-1), \S^{w_1+w_6}(-1)) = \C[-1].
\end{align*}
The first two vanishings follow immediately from Lemma \ref{acyclic}. The third relation follows from expanding the tensor product with Lemma \ref{tensorproducts} and then applying the Borel-Weil-Bott Theorem.

For $\E$ and $\F$, the vanishings follow from Lemma \ref{acyclic} and the Borel-Weil-Bott Theorem. For the exceptionality of $E$, we need to show that
\begin{align*}
    &\ext(\O(1), \E) = 0,\\
    &\ext(\S^{w_1}(1), \E) = \C[0].
\end{align*}

For the first vanishing, note that $\mathbf{L}_{\langle \O(1) \rangle}(\S^{w_1}(1)) = \E$ and thus we have that $\E \in \langle \O(1) \rangle^\perp$. For the second vansihing, replace $\E$ with its semi-simplification to see that
\[\ext(\S^{w_1}(1),\E) = \hh(X,\S^{w_1}(-1)) \oplus \hh(X, \S^{w_1+w_6}(-1)) \oplus \hh(X, \O) = \C[0].\]

The exceptionality of $\F$ follows from a similar calculation.
\end{proof}

\subsection{The main result}
This section is dedicated to proving our main result. Let $X  =E_6/P_2$. We intend to construct an exceptional collection for $\mathbf{D}^b(X)$ of maximal length. We begin with a few intermediary results. 

\begin{lemma}\label{collections} The following collections of vector bundles are exceptional.
\begin{align*}
  &\mathcal{D} = \langle \O, \O(1), \O(2), \ldots, \O(10)\rangle,\\
    &\mathcal{E} = \langle \S^{w_1}, \S^{w_1}(1), \S^{w_1}(2), \ldots, \S^{w_1}(10)\rangle,\\
    &\mathcal{F}= \langle \S^{2w_1}, \S^{2w_1}(1), \S^{2w_1}(2), \ldots, \S^{2w_1}(10)\rangle,\\
    &\mathcal{G}_3 = \langle \S^{3w_1}, \S^{3w_1}(1), \S^{3w_1}(2)\rangle,\\
    &\mathcal{G}_4 = \langle \S^{4w_1}, \S^{4w_1}(1) \rangle.
    \end{align*}
\end{lemma}
\begin{proof}
For $\mathcal{D}$, we can see this is an exceptional sequence due to Lemma \ref{cohomology}. 

For $\mathcal{E}$, we already know the objects are exceptional by Lemma \ref{exceptional}. It remains to check that \[\ext(\S^{w_1}(i), \S^{w_1}) = 0 \text{ for } 1 \leq i \leq 11.\]

We proceed with the computation in detail:
\begin{align*}
    \ext(\S^{w_1}(i), \S^{w_1}) &= \ext(\O, \S^{w_1}(i)^\vee \otimes \S^{w_1})\\
    &= \hh(X, \S^{-w_0^L (w_1+iw_2})) \otimes \S^{w_1})\\
    &= \hh(X, \S^{w_6}(-i-1) \otimes \S^{w_1})\\
    &= \hh(X, \O(-i) \oplus \S^{w_1+w_6}(-i-1))\\
    &=0.
\end{align*}
The vanishing follows from Lemma \ref{tensorproducts} and Lemma \ref{acyclic}, since both vector bundles are acyclic. 

The cases $\mathcal{F},\mathcal{G}_3, \mathcal{G}_4$ are similar. 
\end{proof}
\begin{lemma}\label{first4}
The collection of vector bundles $\mathcal{M}$ with starting block \[\mathbf{M}_0 = \langle \O, \T, \S^{w_1}, \S^{2w_1}\rangle\] supported on the partition $p = (4)$  is a rectangular Lefschetz exceptional collection with respect to $\O(1)$ of length $44$.
\end{lemma}
\begin{proof}
By Definition \ref{lefschetzdef} and Lemma \ref{lefscetzexceptional}, we first show that $\mathbf{M}_0$ is an exceptional collection, i.e. we need to show the vanishing of
\[\ext(\T, \O), \quad \ext(\S^{tw_1},\T), \quad \text{and } \ext(\S^{2w_1}, \S^{w_1}).\]

The first vanishing follows from Lemma \ref{exceptionalblocks}. 

For the third vanishing, we have \[\ext(\S^{2w_1}, \S^{w_1}) = \hh(X, \S^{2w_6}(-2) \otimes \S^{w_1}) = \hh(X, \S^{w_1+2w_6}(-2) \oplus \S^{w_6}(-1)) = 0,\]
by Lemma \ref{acyclic}. 

Secondly we must show that the following extension modules vanish
\begin{align}
    &\ext(\O, V(-i)) = 0, \label{M04}\\
    &\ext(\T, V(-i)) = 0,\label{M05}\\
    &\ext(\S^{w_1}, V(-i)) = 0,\label{M06}\\
    &\ext(\S^{2w_1}, V(-i)) = 0,\label{M07}
\end{align}
for $1 \leq i \leq 10$ and all vector bundles $V \in \mathbf{M}_0$. 

For case \ref{M04}, we have $\ext(\O, V(-i)) = \hh(X, V(-i)) = 0$. This vanishing follows from the acyclicity of the vector bundles $V(-i)$ (see Lemma \ref{exceptional}). For case \ref{M05}, if $V = \O$, the vanishing holds since $\T \in {}^\perp\langle \O \rangle$. If $V = \T$, the vanishing holds by Lemma \ref{exceptionalblocks} from the exceptionality of collection $\mathcal{A}$. Now, if $V = \S^{w_1}, \S^{2w_1}$, we may replace $\T$ with its semi-simplification, and obtain for $t\in \{1,2\}$:
\begin{align*}
    \ext(\T, \S^{tw_1}(-i)) &= \ext(\O \oplus \S^{w_4}(-1) \oplus \O(1), \S^{tw_1}(-i))\\
    &= \ext(\O, \S^{tw_1}(-i)) \oplus \ext(\S^{w_4}(-1), \S^{tw_1}(-i)) \oplus \ext(\O(1), \S^{tw_1}(-i))\\
    &= \hh(X, \S^{w_1}(-i)) \oplus \hh(X, \S^{w_4}(-2) \otimes \S^{w_1}(-i)) \oplus \hh(X, \S^{w_1}(-i-1))\\
    &= 0 \oplus \hh(X, \S^{w_4}(-2) \otimes \S^{w_1}(-i)) \oplus 0\\
    &= \hh(X, \S^{tw_1+w_4}(-i-2)) \oplus \hh(X, \S^{(t-1)w_1+w_5}(-i-1)).
\end{align*}
Here we used the properties of $\mathrm{Ext}$ together with Lemma \ref{tensorproducts} for decomposing tensor products and Lemma \ref{acyclic} to derive the last two lines. This vanishes because $\S^{w_1+w_4}(-i-2)$, $\S^{2w_1+w_4}(-i-2)$, $\S^{w_5}(-i-1)$, and $\S^{w_1+w_5}(-i-1)$ are acyclic for $1 \leq i \leq 10$ by Lemma \ref{acyclic}. 

For cases \ref{M06} and \ref{M07}, we need $\ext(\S^{tw_1}, V(-i)) = 0$. For $V=\O$, this follows immediately from Lemma \ref{acyclic}. For $V = \T$, using the cohomology of the complex from Lemma \ref{cohomologyextension}, we get that $\ext(\S^{tw_1}, \T) = 0$ is implied by 
\begin{align}
    &\ext(\S^{tw_1}(i), \O) = 0 &&\text{ for } 1 \leq i \leq 10, \label{M08}\\
    &\ext(\S^{tw_1}(i), \S^{w_1+w_6}(-1)) = 0 &&\text{ for } 1 \leq i \leq 10,\label{M09}\\
    &\ext(\S^{tw_1}(i), \T(-1)) = 0 &&\text{ for } 1 \leq i \leq 10. \label{M010}
\end{align}

Equation \ref{M08} follows immediately from Lemma \ref{acyclic}. Using Lemma \ref{tensorproducts}, equation \ref{M09} can be rewritten as:
\begin{align*}
    \ext(\S^{tw_1}(i), \S^{w_1+w_6}(-1))&= \hh(X, \S^{tw_6}(-i-t) \otimes (\S^{w_1+w_6}(-1))\\
    &= \hh(X, \S^{tw_6}(-i-t) \otimes \S^{w_1+w_6}(-1)))\\
    &= \hh(X, \S^{w_3+tw_6}(-i-t)) \oplus \hh(X, \S^{w_1+(t-1)w_6}(-i-t+1))\\
    &\oplus \hh(X, \S^{w_4+tw_6}(-i-t)) \oplus \hh(X, \S^{w_3 + (t-1)w_6}(-i-t+1)).
\end{align*}

The vanishing follows from the fact that $\S^{w_3+w_6}(-i-1)$, $\S^{w_3+2w_6}(-i-2)$, $\S^{w_4+w_6}(-i-1)$, $\S^{w_4+2w_6}(-i-2)$, $\S^{w_1}(-i)$, $\S^{w_1+w_6}(-i-1)$, $\S^{w_3}(-i)$, and $\S^{w_3 + w_6}(-i-1)$ are all acyclic for $1 \leq i \leq 10$ by Lemma \ref{acyclic}. For equation \ref{M010}, Lemma \ref{dualextension} gives $\ext(\S^{tw_1}(i), \T(-1)) = \ext(\T, \S^{tw_1}(-i-t))$. This has already been computed to prove equation \ref{M05}. Lemma \ref{acyclic} gives the required vanishing. 

For $V = \S^{uw_1}$ and $t=u$, the vanishing of equations \ref{M06} and \ref{M07} is given by Lemma \ref{exceptional}, since $\mathcal{F}_1$ and $\mathcal{F}_2$ are exceptional collections. It remains to show that
\begin{align}
    &\ext(\S^{w_1}, \S^{2w_1}(-i)) = 0 &&\text{ for } 1 \leq i \leq 10, \label{M011}\\
    &\ext(\S^{2w_1}, \S^{w_1}(-i)) = 0 &&\text{ for } 1 \leq i \leq 10. \label{M012}
\end{align}

This follows by Lemma \ref{acyclic} from the acyclicity of $\S^{w_1+2w_6}(-i-2)$, $\S^{2w_1+w_6}(-i-1)$, $\S^{w_6}(-i-1)$, and $\S^{w_1}(-i)$ for $1 \leq i \leq 10$. This concludes the proof.

\end{proof}

\begin{lemma}\label{first6}
The collection of vector bundles $\mathcal{N}$ with starting block \[\mathbf{N}_0 = \langle \O, \T, \S^{w_1}, \S^{2w_1}, \E,\F \rangle\] and partition $p = (6)$ is a rectangular Lefchsetz exceptional collection of length $66$. 
\end{lemma}
\begin{proof}
Lemma \ref{first4} shows that four of the vector bundles form a Lefschetz exceptional collection, essentially merging collections $\mathcal{D}$, $\mathcal{A}$, $\mathcal{E}$, and $\mathcal{F}$. It now remains to show that collection $\mathcal{M}$ can be merged with collections $\mathcal{B}$ and $\mathcal{C}$ as intended. To do this, we need to show
\begin{align}
    &\ext(\E, V(-i)) = 0 \text{ for } 1 \leq i \leq 10 \text{ and } V \in \mathbf{N}_0,\label{mix1}\\
    &\ext(\F, V(-i)) = 0 \text{ for } 1 \leq i \leq 10 \text{ and } V \in \mathbf{N}_0. \label{mix2}
\end{align}

\underline{Equation \ref{mix1}.} For $V = \O$, this was proved in Lemma \ref{exceptionalblocks}.

For $V=\T$, using the cohomology of the complex from Lemma \ref{cohomologyextension}, the vanishing of $\ext(\E(i), \T)$ is equivalent to 
\begin{align}
    &\ext(\E(i), \O) = 0 \text{ for } 1 \leq i \leq 10,\label{TE1}\\
    &\ext(\E(i), \S^{w_1+w_6}(-1)) = 1 \text{ for } 1 \leq i \leq 10,\label{TE2}\\
    &\ext(\E(i), \T(-1)) = 0 \text{ for } 1 \leq i \leq 10.\label{TE3}
\end{align}
The first of these equations was proved just before. For the second, we use Lemma \ref{tensorproducts} and the semi-simplification of $\E$ to see that the vanishing is equivalent to the acyclicity of $\S^{w_1+w_6}(-i-3)$ and $\S^{w_6}(-i-2)$. For \ref{TE3}, we can dualise both terms to see $\ext(\E(i), \T(-1)) = \ext(\T, \E^\vee (-i))$. Replacing $\T$ with its semi-simplification this term vanishes when 
\begin{align}
    &\ext(\T, \O(-i-1)) = 0,\\
    &\ext(\T, \S^{w_6}(-i-2)) = 0.
\end{align}

The former was proved in Lemma \ref{exceptionalblocks}. The latter follows from the acyclicity of $\S^{w_6}(-i-2)$, $\S^{w_4+w_6}(-i-4)$, and $\S^{w_3}(-i-3)$. All these cases are covered by Lemma \ref{acyclic}. 

$\xqed$

For $V=\E$, the vanishing is a consequence of collection $\mathcal{B}$ in Lemma \ref{exceptionalblocks} being exceptional. 

$\xqed$

For $V = \F$, Lemmas \ref{cohomologyextension} and \ref{tensorproducts} give that the vanishing is equivalent to
\begin{align}
    &\ext(\E(i), \O) = 0 \text{ for } 1 \leq i \leq 10,\label{EF1}\\
    &\ext(\E(i), \S^{w_6}(1)) = 0 \text{ for } 1 \leq i \leq 10,\label{EF2}.
\end{align}

Equation \ref{EF1} follows from acyclicity of $\O(-i-1)$ and $\S^{w_1}(-i-2)$. For equation \ref{EF2}, replace $\E$ with its semi-simple parts and calculate with Lemmas \ref{dualextension} and \ref{tensorproducts}:
\begin{align}
    &\ext(\O(i+1), \S^{w_6}(1)) = 0 \text{ for } 1 \leq i \leq 10,\label{EF4}\\ 
    &\ext(\S^{w_1}(i+1), \S^{w_6}(1)) = 0  \text{ for } 1 \leq i \leq 10,\label{EF5}.
\end{align}

Equation \ref{EF4} follows from the acyclicity of $\S^{w_6}(-i)$. \ref{EF5} follows from the acyclicity of $\S^{2w_6}(-i-1)$ and $\S^{w_5}(-i-1)$. These are covered by Lemma \ref{acyclic}. 
$\xqed$

For $V = \S^{tw_1}$, replace $\E$ with its semi-simple parts. Then vanishing is equivalent to

\begin{align}
    &\ext(\O(1), \S^{tw_1}(-i)) = \text{ for } 1 \leq i \leq 10,\label{ES1}\\
    &\ext(\S^{w_1}(1), \S^{tw_1}(-i)) = \text{ for } 1 \leq i \leq 10.\label{ES2}
\end{align}

Equation \ref{ES1} is equivalent to the acyclicity of the vector bundle $\S^{tw_1}(-i-1)$. Equation \ref{ES2} can be rewritten as
\begin{gather*}
    \ext(\S^{w_1}(1), \S^{tw_1}(-i)) = \hh(X, \S^{w_6}(-2) \otimes \S^{tw_1}(-i-1)) \\= \hh(X, \S^{w_6+tw_1}(-i-3) \oplus \S^{(t-1)w_1}(-i-2)).
\end{gather*}

Lemma \ref{acyclic} covers all such cases.
$\xqed$

\underline{Equation \ref{mix2}.} For $V=\O$, the vanishing was studied in Lemma \ref{exceptionalblocks}.

For $V = \T$, substituting $\F$ and $\T$ with their semi-simplification, the vanishing of $\ext(\F(i), \T)$ is equivalent to 
\begin{align}
    &\ext(\F(i), \O) = 0 \text{ for } 1 \leq i \leq 10,\label{FT1}\\
    &\ext(\F(i), \S^{w_1+w_6}(-1)) = 1 \text{ for } 1 \leq i \leq 10,\label{FT2}\\
    &\ext(\F(i), \T(-1)) = 0 \text{ for } 1 \leq i \leq 10.\label{FT3}
\end{align}

Equation \ref{FT1} was computed in Lemma \ref{exceptionalblocks}. Equation \ref{FT2} is equivalent to the acyclicity of $\S^{w_1+2w_6}(-i-3)$, $\S^{w_6}(-i-2)$, and $\S^{w_1+w_5}(-i-3)$. Similar to equation \ref{TE2}, it remains to show that \[\ext(\T, \S^{w_1}(-i-2)) = 0.\]

Replacing $\T$ with its semi-simplification, this vanishing is equivalent to the following vector bundles being acyclic: $\S^{w_1}(-i-2)$, $\S^{w_1+w_4}(-i-4)$, and $\S^{w_5}(-i-3)$. 
$\xqed$

For $V=\F$, vanishing is a consequence of collection $\mathcal{C}$ in Lemma \ref{exceptionalblocks} being exceptional.

$\xqed$

For $V=\S^{tw_1}$, replace $\F$ with its semi-simple parts. Then vanishing is equivalent to

\begin{align}
    &\ext(\O(1), \S^{tw_1}(-i)) = \text{ for } 1 \leq i \leq 10,\label{FS1}\\
    &\ext(\S^{w_6}(1), \S^{tw_1}(-i)) = \text{ for } 1 \leq i \leq 10. \label{FS2}
\end{align}

Equation \ref{FS1} follows from the acyclicity of $\S^{tw_1}(-i-1)$. Equation \ref{FS2} can be rewritten as
\[\ext(\S^{w_6}(1), \S^{tw_1}(-i)) = \hh(X, \S^{w_1}(-2)\otimes \S^{tw_1}(-i)) = \hh(X, \S^{(t+1)w_1}(-i-2) \oplus \S^{w_3+(t-1)w_1}(-i-1)).\] 

Finally, Lemma \ref{acyclic} covers all these cases.

$\xqed$

\underline{The starting block.} 

By Definition \ref{lefschetzdef}, it remains to show some vanishings between the vector bundles in the starting block, namely: $\ext(\E, \O)$, $\ext(\E, \T)$, $\ext(\E, \S^{tw_1})$, $\ext(\F, \O)$, $\ext(\F, \T)$, $\ext(\F, \E)$, and $\ext(\F, \S^{tw_1})$ for $t \in \{1,2\}$. The cases $\ext(E,\O)$ and $\ext(F,\O)$ are immediate to check using semi-simplifications. The remainder can be confirmed by setting $i=0$ in previous calculations. 
\end{proof}
\begin{theorem}
The collection of vector bundles $\mathcal{Q}$ with starting block 
\[\mathbf{Q}_0 = \langle \O, \T, \S^{w_1}, \S^{2w_1}, \E, \F, \S^{3w_1}, \S^{4w_1}, \S^{5w_1} \rangle\]
and partition $p = [9, 8, 7, 6, 6, 6, 6, 6, 6, 6, 6]$ is an exceptional collection for $\mathbf{D}^b(X)$ with respect to $\O(1)$ of maximal length equal to $72$. 
\end{theorem}
\begin{proof}
In view of Lemmas \ref{first4} and \ref{first6}, we merge the collection $\mathcal{N}$ with $\mathcal{G}_3$, $\mathcal{G}_4$ and the exceptional bundle $\S^{5w_1}$. We have to prove the following vanishings:

\begin{align}
    &\ext(\S^{5w_1}, V) = 0 \text{ for } V \in \mathbf{Q}_0 \setminus \{\S^{5w_1}\},\label{FIN0}\\
    &\ext(\S^{4w_1}, V(-i)) = 0 \text{ for } i\in \{0,1\} \text{ and } V \in \mathbf{Q}_0,\label{FIN1}\\
       &\ext(\S^{3w_1}, V(-i)) = 0 \text{ for } i\in \{0,1,2\} \text{ and } V \in \mathbf{Q}_0,\label{FIN2}\\
       &\ext(V, \S^{tw_1}(-i)) = 0 \text{ for } 1\leq i \leq 10, t \in \{3,4,5\} \text{ and } V \in \mathbf{Q}_0. \label{FIN3}
\end{align}

\underline{Equations \ref{FIN0}, \ref{FIN1}, and \ref{FIN2}.}

For $V = \O$, vanishing follows from the acyclicity of $\S^{tw_6}(-i-t)$ for $t \in \{3,4,5\}$ and $i \in \{0,1,2\}$. For $V = \T$, use the complex from Lemma \ref{cohomologyextension} to deduce that the vanishing is equivalent to 

\begin{align}
    &\ext(\S^{tw_1}(i), \O) = 0 \text{ for } 0 \leq i \leq 2,\label{finn1}\\
    &\ext(\S^{tw_1}(i), \S^{w_1+w_6}(-1)) = 0 \text{ for } 0 \leq i \leq 2,\label{finn2}\\
    &\ext(\S^{tw_1}(i), \T(-1)) = 0 \text{ for } 0 \leq i \leq 2.\label{finn3}
\end{align}

Equation \ref{finn1} follows from the acyclicity of $\S^{tw_1}(-i)$ for $t \in \{3,4,5\}$. Equation \ref{finn2} is equivalent to the acyclicity of $\S^{w_1+(t+1)w_6}(-i-t-1)$, $\S^{w_1+w_5+(t-1)w_6}(-i-t-1)$, $\S^{tw_6}(-i-t)$, and $\S^{w_5+(t-2)w_6}(-i-t)$. For equation \ref{finn3}, replace $\T$ with its semi-simple parts. The vanshing is then equivalent to the acyclicity of $\S^{w_4+tw_6}(-i-t-2)$ and $\S^{w_3+(t-1)w_6}(-i-t-1)$.

For $V = \E$, replace $\E$ with its semi-simple parts to obtain
\[\ext(\S^{tw_1}, \E(-i)) = \hh(X,\S^{tw_6}(-i-t))\oplus \hh(X, \S^{tw_6}(-t)\otimes \S^{w_1}(1)).\]

The vanishing is then equivalent to the acyclicity of $\S^{tw_6}(-i-t+1)$, $\S^{tw_6+w_1}(-i-t+1)$, and $\S^{(t-1)w_6}(-i-t)$ for $t \in \{3,4,5\}$. For $V = \F$, we follow a similar procedure to see that the vanishing is equivalent to the acyclicity of $\S^{w_3+tw_6}(-i-t-1)$ and $\S^{w_1+(t-1)w_6}(-i-t)$ for $t \in \{3,4,5\}$.

The cases $\ext(\S^{tw_1}, \S^{tw_1}(-i)) = 0$ for $t \in \{3,4\}$ and $i \in \{0,1,2\}$ follow from Lemma \ref{collections} as $\mathcal{F}_3$ and $\mathcal{F}_4$ are both exceptional collections.

The case $\ext(\S^{4w_1}, \S^{3w_1}(-i)) = 0$ for $i \in \{0,1\}$ follows from the acyclicity of $\S^{4w_6+3w_1}(-i-4)$, $\S^{3w_6+2w_1}(-i-3)$, $\S^{2w_6+w_1}(-i-2)$, and $\S^{w_6}(-i-1)$. Finally, the case $\ext(\S^{3w_1}, \S^{4w_1}(-i)) = 0$ for $i \in \{1,2\}$ follows from the acyclicity of $\S^{3w_6+4w_1}(-i-3)$, $\S^{2w_6+3w_1}(-i-2)$, $\S^{w_6+2w_1}(-i-1)$, $\S^{w_1}(-i)$ for $i \in \{1,2\}$. All cases are covered by Lemma \ref{acyclic}. 

$\xqed$

\underline{Equation \ref{FIN3}.}

For $V = \O$, vanishing is immediate by Lemma \ref{acyclic}. For $V = \T$, calculations are identical to equation \ref{M05}. For $V = \E$, calculations are identical to equation \ref{ES2}. For $V = \F$, calculations are identical to equation \ref{FS2}. For $V=\S^{w_1}$ and $\S^{2w_1}$, calculations are similar to equations \ref{M011} and \ref{M012}. Lemma \ref{acyclic} covers all necessary cases, which concludes the proof. 
\end{proof}

\bibliography{article}
\bibliographystyle{acm}
\end{document}